\documentclass{article}

\usepackage[T2A]{fontenc}
\usepackage[utf8]{inputenc}
\usepackage[english]{babel}

\usepackage{amsmath,amssymb,mathrsfs}
\usepackage{amsthm}
\usepackage{hyperref}
\usepackage[hang,flushmargin]{footmisc}
\usepackage{color}

\theoremstyle{plain}
\newtheorem{theorem}{Theorem}[section]
\newtheorem*{theorem*}{Theorem}
\newtheorem{corollary}{Corollary}[section]
\newtheorem{lemma}{Lemma}[section]
\newtheorem*{lemma*}{Lemma}
\newtheorem{proposition}{Proposition}[section]

\theoremstyle{definition}

\renewcommand{\Re}{\mathop{\mathrm{Re}}}

\DeclareMathOperator{\supp}{supp}

\DeclareMathOperator{\lspan}{span}

\begin{document}

\begin{center}
    \Large\bfseries On the injectivity of the generalized Radon transform arising in a model of mathematical economics
\end{center}

\begin{center}\it A.\,D. Agaltsov\footnote{Centre de Mathématiques Appliquées, Ecole Polytechnique, Route de Saclay, 91128 Palaiseau, France; email: agaltsov@cmap.polytechnique.fr}\footnote{Moscow State University, GSP-1, Leninskie Gory, 119991 Moscow, Russia}
\end{center}

\begin{quote}
In the present article we consider the uniqueness problem for the generalized Radon transform arising in a mathematical model of production. We prove uniqueness theorems for this transform and for the profit function in the corresponding model of production. Our approach is based on the multidimensional Wiener's approximation theorems.\medskip

\textbf{Keywords:} generalized Radon transform, production theory, uniqueness theorem.

\textbf{AMS classification}: 35R30 (inverse problems), 44A12 (Radon transform), 46N10 (applications of functional analysis in economics)
\end{quote}

\section{Introduction}
\textbf{The problem of microfoundations} is an important problem of modern macroeconomics. Early macroeconomic models were based on certain assumptions about macroeconomic variables. This approach aroused much controversy over the consistency of these assumptions with the laws of microeconomics. This led to the creation of macroeconomic models based on aggregation of microdescriptions.

In the theory of production functions the first model of this type was proposed in \cite{Houth1955}. It is shown in \cite{Houth1955} that the Cobb--Douglas production function arises as an aggregate production function of an industry that consists of production units with fixed proportion production technologies, in the case when the resources are distributed according to the Pareto law. Using the model of the article \cite{Houth1955} as a basis, L. Johansen in \cite{Joh1972} proposed a general framework for construction of production functions for industries representable as a union of production cells with fixed proportion production technologies. Some years later, it was shown in \cite{Henk1990} that under some mild conditions, the production function of an industry has at most one micro-founded description in this framework, and this description can be obtained using an explicit formula.

This framework has proven to be a successful tool for studying and understanding macroeconomic processes in such countries as Norway, Sweden, USA, India, Japan and others, see, e.g., \cite{Joh1972,Petr1979en,Petr1996en}. In particular, this model has been shown to be well adapted to take into account the changes in the production process caused by scientific and technological progress, see \cite{Joh1972}. Hovewer, the scope of this model is limited to production systems satisfying certain strict criteria in terms of their profit functions, see \cite{Henk1990}. 

In the 1990's--2000's the world economy faced the challenge of globalization. In this process producers become inserted into global markets via the export of goods and import of resources and technologies. Furthermore, the difference in inflation rates in internal and global markets leads to a situation where the proportion of imported resources used in production constantly changes. This leads to the need of revision and modification of aggregate models of production based on the assumption of fixed proportion production technologies at the microlevel. For a more detailed discussion of the impact of globalization on the microeconomics and its relation to aggregate economic models see, e.g., \cite{Tayl2001}.

A revision and modification of the model of \cite{Houth1955,Joh1972} was fulfiled in \cite{Sato1975,Shan1997aen,Shan1997ben}. The modified model (which we will refer to as the generalized Houthakker--Johansen model) provides a framework for construction of production functions for industries which are representable as a union of production units with neo-classical production technologies, allowing substitution between factors. The scope of the generalized Houthakker--Johansen model was investigated in \cite{Agal2015den}. It was shown in \cite{Agal2015den} that industries fitting the framework are characterized by an easily verifiable condition in terms of their profit functions.

The present article is devoted to developing the theoretical basis for the generalized Houthakker--Johansen model. We show, in particular, that the production function of an industry has at most one micro-founded description in the framework of the generalized Houthakker--Johansen model if and only if the production technologies satisfy a certain condition. This condition is fulfilled, in particular, for constant elasticity of substitution (CES) production technologies, including the case of fixed proportion technologies. Furthermore, this condition is stable with respect to aggregation of production factors at the micro-level (i.e., with respect to ``composition'' of production technologies). Our results can be considered as a development of the results of \cite{Henk1990}, where the case of fixed proportion technologies (corresponding to the model of \cite{Houth1955,Joh1972}) was considered.

We expect that the generalized Houthakker--Johansen model can be fruitfully used for studying macroeconomic processes under conditions of globalization. In particular, we expect it to be well suited for studying the implications of technological innovation on the production process. This question of determining the explanatory potential of the generalized Houthakker--Johansen model in applications to real data is of great interest, and it requires a subsequent research.

\textbf{Generalized Houthakker--Johansen model.} We consider a model of production introduced in \cite{Sato1975,Shan1997aen,Shan1997ben}. In this model, following \cite{Houth1955}, the industry is represented as a union of production cells (which can be considered as individual firms, machines or branches of industry). Each of these production cells can produce the same final product by means of $n \geq 2$ inputs using some technology. The model is described by two functions $f$, $F_0$ defined below.

The technologies are parametrized by vectors $x \in \mathbb R^n_+ = \bigl\{ (x_1,\ldots,x_n) \mid \forall k, \; x_k > 0 \bigr\}$. For each technology $x$ we define its capacity $f(x) \geq 0$ (which can be considered as a number of production cells using this technology). The function $f \colon \mathbb R^n_+ \to \mathbb R$ is called the distribution of capacities over technologies.

The function $F_0 \colon \mathbb R^n_+ \to \mathbb R^1_+$ is called the production function at the micro-level. Using $F_0$ we associate to each technology $x$ its production function~$F_x$:
\begin{equation*}
    F_x(u_1,\ldots,u_n) = \min \bigl\{ 1, F_0\bigl(\tfrac{u_1}{x_1},\ldots,\tfrac{u_n}{x_n} \bigr) \bigr\},
\end{equation*}
where $u = (u_1,\ldots,u_n)$ is the vector of inputs of the (unit) production cell using technology $x$ and $F_x(u_1,\ldots,u_n)$ is its output for one production period.

We suppose that $F_0$ has the neoclassical properties. It means that it is positive homogeneous of degree one (scaling of inputs corresponds to the same scaling of output), increasing in each variable (all inputs contribute to output), concave and continuous. The second property implies, in particular, that different inputs can substitute each other to a certain extent in the production process. This is typical for production systems experiencing the effects of globalization and standartization.

Given total inputs $l = (l_1,\ldots,l_n) \in \mathbb R^n_+$ for the union of the production cells (the industry), we seek for the maximal total output that we can obtain by varying the distribution of inputs $u = (u_1,\ldots,u_n)$ over the production cells:
\begin{equation}\label{in.optprob}
   F_A(l) = \max_{u} \left\{ \int_{\mathbb R^n_+} F_x\bigl(u(x)\bigr) f(x) \, dx \; \middle| \; \forall k, \int_{\mathbb R^n_+} u_k (x)f(x) \, dx \leq l_k \right\}. 
\end{equation}
The optimization problem \eqref{in.optprob} always has a solution in the class of non-negative measurable $u = (u_1,\ldots,u_n)$ with $u_1 f$, \dots, $u_n f \in L^1(\mathbb R^n_+)$ (it is proved in \cite{Shan1997aen}). We call $F_A$ the aggregate production function of the industry.

It can be shown that optimal distributions of resources in optimization problem \eqref{in.optprob} are provided by market mechanisms. More precisely, optimal distributions can be obtained by assuming that production units independently maximize their profits for some fixed prices of resources and final product, see \cite[Theorem 4.1]{Shan1997aen}.

However, the assumption that each production unit maximizes its profit is a simplification of the actual situation, especially under conditions of globalization. The global market is a highly interconnected system, and an excessive focus of individual firms within corporations on rewards together with risk compartmentalization (i.e. a situation when individual agents become unaware of implications of their actions for the whole system) can lead to the collapse of the system as a whole. The global financial meltdown of the 2000's is an example of this effect. For a detailed discussion of the subject, see \cite[Chapter 10]{Ferr2011}.

Using the aggregate production function $F_A$ we can compute the maximal possible profit of the industry for one production period:
\begin{equation*}
    (\Pi_q f)(p_0,p) = \sup \bigl\{ p_0 F_A(l) - p \cdot l \mid l \in \mathbb R^n_+ \bigr\},
\end{equation*}
where $p_0$ and $p=(p_1,\ldots,p_n)$ are the unit prices of the output and of the inputs, respectively. Function $\Pi_q$ is called the profit function of the industry. It can be shown that
\begin{equation}\label{in.prodfun}
      (\Pi_q f)(p_0,p) = \int_{\mathbb R^n_+} \max\bigl\{0,p_0 - q(p_1 x_1,\ldots,p_n x_n) \bigr\} f(x) \, dx,
\end{equation}
where $q(p_1 x_1,\ldots, p_n x_n)$ is the unit cost of production at the (unit) cell using technology $x=(x_1,\ldots,x_n)$ and $f$ is the distribution of capacities over technologies. The unit cost function $q$ is related to the production function at the micro-level by the following transform:
\begin{equation}\label{in.F0toq}
    q(x) = \inf\bigl\{ {\textstyle \frac{x \cdot y}{F_0(y)}} \mid y \in \mathbb R^n_+, \; F_0(y)>0 \bigr\}, \quad x \in \mathbb R^n_+.
\end{equation}
It inherits such properties of neoclassical production functions as positivity and positive homogeneity of degree one. Thus, in the present article we assume that
\begin{equation}\label{in.defq}
    \begin{gathered}
        \text{\it $q \in C^1(\mathbb R^n_+)$, $q(x) > 0$ and $q(\lambda x) = \lambda q(x)$ for $\lambda > 0$, $x \in \mathbb R^n_+$.} 
    \end{gathered}
\end{equation}

Note that the second derivative of function $\Pi_q f$ of \eqref{in.prodfun} with respect to $p_0$ is the generalized Radon transform
\begin{equation}\label{in.defRq}
    (R_q f) (p) = \int_{q_p^{-1}(1)} f(x) \, \frac{dS_x}{|\nabla q_p(x)|}, \quad p = (p_1,\ldots,p_n) \in \mathbb R^n_+, \\
\end{equation}
where $q_p(x) = q(p_1 x_1,\ldots,p_n x_n)$, $\nabla$ is the standard gradient in variable $x = (x_1,\ldots,x_n)$, $dS_x$ is the hypersurface measure on $q_p^{-1}(1) = \bigl\{ x \in \mathbb R^n_+ \colon q_p(x) = 1 \bigr\}$. It follows from \eqref{in.defq} that $\nabla q_p(x)$ doesn't vanish for $p$, $x \in \mathbb R^n_+$. 

Also note that the profit function \eqref{in.prodfun} can be written in the form $(\Pi_q f)(p_0,p) = (R^h_q f)(p)$ with $h(t) = \max\{0,p_0 - t\}$, where 
\begin{equation}\label{in.defRhq}
        (R^h_q f)(p) = \int_{\mathbb R^n_+} h(q(p_1 x_1,\ldots,p_n x_n)) f(x) \, dx, \quad p \in \mathbb R^n_+.
\end{equation}

The production function $F_A$ and the profit function $\Pi_q$ are the main (and equivalent) tools of description of the industry at the macro-level (i.e. as a whole). It follows from formulas \eqref{in.prodfun} and \eqref{in.F0toq} that $\Pi_q$ is completely determined by the micro-level information (i.e. information concerning production cells as independent units) $F_0$ and $f$.

\textbf{Production technologies.} As it was mentioned above, the generalized Houthakker--Johansen model allows arbitrary neo-classical production functions (technologies) at the micro-level. Hovewer, technologies are usually approximated using some standard functions with parameters identified from statistical data. The choice of these standard functions depends on the extent to which the production factors can substitute each other. 

The elasticity of substitution is the most used quantitative characteristic of substitutability of production factors, introduced in \cite{Hick1932} for the case of two factors. However, in general this characteristic is too difficult to evaluate from statistical data. In practice, one usually makes an assumption of constant elasticity of substitution of factors. The case of zero elasticity of substitution corresponds to fixed proportion production functions. On the other hand, the case of unitary elasticity of substitution corresponds to Cobb--Douglas production functions, introduced in \cite{Doug1928} for the case of two factors (namely, labor and capital). The Cobb--Douglas production functions are the simpliest and the most widely used in economics production functions allowing substitution between inputs. However, empirical evidence shows that these functions do not provide a satisfactiry description for many production systems. For example, some strong indications that the elasticity  of substitution between labor and capital in manufacturing can be less than one were pointed in \cite{Arro1961}.

A general production function with constant elasticity of substitution of production factors (CES) was introduced in \cite{Arro1961} for the case of two factors. Possible generalizations to an arbitrary number of factors were proposed by Allen and Hicks, Uzawa, McFadden, Morishima and others, see, e.g., \cite{Fron2011} for a brief survey. However, there is the unique class of production functions $F_0$ with several inputs for which all these elasticities are constant. These production functions $F_0$ correspond, according to \eqref{in.F0toq}, to the unit cost functions $q$ of the form $q = q_\alpha$, $\alpha \in [-\infty,1]$, where
\begin{equation}\label{in.defCES}
    \begin{aligned}
    q_\alpha(x) & = C ( a_1 x_1^\alpha+\cdots+ a_n x_n^\alpha)^{\frac 1 \alpha}, \quad \alpha \in (-\infty,1] \setminus 0,\\
    q_{-\infty}(x) & = C \min(a_1 x_1,\ldots, a_n x_n), \\
    q_0(x) & = C x_1^{a_1} \cdots x_n^{a_n},
    \end{aligned}
\end{equation}
and $C$, $a_1$, $\ldots$, $a_n > 0$, $a_1+\cdots+a_n=1$. The CES production technologies have become very popular in production theory (especially, in applied analysis) where they replace the Cobb--Douglas production functions. Hovewer, a significant drawback of CES technologies is that they imply the same elasticity of substitution for any pair of production factors.

Sato was the first who proposed to consider nested CES technologies allowing different degrees of substitutability between production factors in different groups, see \cite{Sato1967}. These technologies allow, in particular, to take into account that the elasticity of substitution between capital and unskilled labor is usually higher than between capital and skilled labor (this effect is known as capital-skill complementarity and it was first formalized in \cite{Gril1969}).

In the present article we consider a more general class of production technologies described by cost functions $q$ of the form \eqref{in.defq} satisfying the following condition:
\begin{equation}\label{in.lvlqbnd}
    \text{\it the level sets of $q$ are bounded.}
\end{equation}
Note that the class of cost functions $q$ satisfying \eqref{in.defq} and \eqref{in.lvlqbnd} contains linear functions with positive coefficients, CES functions with elasticity of substitution less than one and is closed under composition in the following sense: if $q \colon \mathbb R^k_+ \to \mathbb R^1_+$, $\phi \colon \mathbb R^m_+ \to \mathbb R^1_+$, $k \geq 2$, $m \geq 2$, both satisfy \eqref{in.defq}, \eqref{in.lvlqbnd}, and $i \in \{1,\ldots,k\}$, then the function $\widetilde q$ defined as
\begin{equation}\label{in.defnq}
  \begin{gathered}
    \widetilde q(x_1,\ldots,x_{i-1},x_{i+1},\ldots,x_k,y) = q(x_1,\ldots,x_{i-1},\phi(y),x_{i+1},\ldots,x_k), \\
    x = (x_1,\ldots,x_k) \in \mathbb R^k_+, \quad y \in \mathbb R^m_+,
  \end{gathered}
\end{equation}
also satisfies \eqref{in.defq}, \eqref{in.lvlqbnd}. The process of passing from cost function $\widetilde q$ to $q$ can be considered as an aggregation of production factors $y$ using cost index $\phi$. It is worth noting that conditions allowing representation of function $\widetilde q$ in the form \eqref{in.defnq} were studied in \cite{Leon1947}. 

\textbf{Outline of the results.} In the present article we are interested in the inverse problem of recovering the micro-level information from the macro-level information in the generalized Houthakker--Johansen model. 

We continue studies of \cite{Sato1975,Henk1990,Shan1997aen,Shan1997ben,Agal2014aen,Agal2015den}. We investigate the conditions under which the production (or profit) function of an industry with production technologies described by \eqref{in.defq}, \eqref{in.lvlqbnd} uniquely determines the distribution of capacities over technologies in the framework of the generalized Houthakker--Johansen model. In \cite{Henk1990} the case of technologies with complementary inputs, corresponding to $q = q_1$, was considered. In \cite{Agal2014aen} the case of CES technologies with degrees of substitutability varying between zero (fixed proportion technologies) and unitary elasticity (Cobb--Douglas technologies), correspoding to $q = q_\alpha$, $\alpha \in (0,1)$, was investigated. In both cases the micro-founded description of the aggregate production function is unique under some mild conditions.

In the present article we give a complete characterization of production systems with technologies described by cost functions $q$ satisfying  \eqref{in.defq}, \eqref{in.lvlqbnd}, for which the aggregate production (or profit) function uniquely determines the distribution of capacities over technologies. More precisely, we show that the aggregate production function has the unique micro-founded description if and only if the Mellin transform of $e^{-q}$ does not vanish on a certain plane. We also show that this is equivalent to injectivity of the generalized Radon transform $R_q$, see Theorem \ref{thm.injPiq} of Section \ref{sec.mr}.

As an application of the aforementioned results, we show that the aggregate production function corresponding to a nested CES cost function at the micro-level has the unique micro-founded description in the generalized Houthakker--Johansen model, see Corollary \ref{cor.nCES} of Section \ref{sec.mr}.

Mathematically, we obtain necessary and sufficient conditions for $q$ (resp. $q$ and $h$) such that the operators $R_q$ and $\Pi_q$ (resp. $R^h_q$) are injective in $L^1(\mathbb R^n_+)$, $L^2(\mathbb R^n_+)$, $L^\infty(\mathbb R^n_+)$ and in some more general (weighted) spaces, see Theorem \ref{thm.injPiq} (resp. \ref{thm.injRhq}) of Section \ref{sec.mr}. 

One can emphasize the three must known techniques used to obtain the uniqueness conditions for Radon transforms. It is sometimes possible to relate a generalized Radon transform to a well-known integral transform like Fourier, Abel or Mellin transforms. This approach, in particular, allows to obtain explicit inversion formulas, and it was used in the fundamental articles on the integral geometry \cite{Funk1916,Rad1917}.

In the case of real-analytic Radon transforms a microlocal approach can be used, see, e.g., \cite{Quin2001,Kris2009}. The idea is that real-analytic Radon transforms propagate analytic singularities in a well-controlled fashion, and one can estimate the support of a function from its analytic wavefront (using the so-called Kashiwara's watermelon theorem). 

In the case of Radon transforms over the curves (more precisely, over the geodesics of some metric) one can also reduce the uniqueness problem to the study of a certain transport equation and energy (Pestov-type) estimates for it, see, e.g., \cite{Shar1994} and \cite[Section 7]{Ferr2009}. The latter approach goes back to R. G. Mukhometov.

In the present article we combine the first aforementioned approach with the multidimensional Wiener's approximation theorems to relate the injectivity of $R_q$ and $R^h_q$ to the absense of zeros of the Mellin transforms of functions $e^{-q}$ and $h$. The idea to use the Wiener's approximation theorems was inspired by the article \cite{Frid1995}. It is still an open question whether zeros of the Mellin transform of function $e^{-q}$ have some meaningful economic interpretation per se.

The main results of the present article are formulated in Section \ref{sec.mr}. In Section \ref{sec.ps} we prove an auxilary proposition relating operators $R_q$, $\Pi_q$, $R^h_q$ to the Mellin transform. This proposition is analogous to the classical projection theorem for the classical Radon transform. In Section \ref{sec.aux} we state and generalize the classical Wiener approximation theorems to the case of the multidimensional Mellin transform. In Section \ref{sec.tPiq} we prove the uniqueness theorem for operators $R_q$ and $\Pi_q$, and in Section \ref{sec.tRhq} we prove the corresponding result for operators $R^h_q$. In Section \ref{sec.pn} we prove the proposition concerning ``the stability'' of the property of injectivity for operators $R_q$, $\Pi_q$ with respect to composition of the form \eqref{in.defnq}; we conclude that operators $R_q$, $\Pi_q$ are injective for nested CES functions $q$.

\section{Main results}\label{sec.mr}
Given two vectors $a=(a_1,\ldots,a_n)$ and $b = (b_1,\ldots,b_n)$ we set $a^b = a_1^{b_1} \cdots a_n^{b_n}$. For a given function $f$ on $\mathbb R^n_+$ we define the norms $\|f\|_{r,c}$, $1 \leq r \leq \infty$, $c \in \mathbb R^n_+$, in the following way:
\begin{gather*}
  \|f\|_{r,c} = \biggl( \int_{\mathbb R^n_+} |f(x)|^r x^{rc-I} dx \biggr)^{1/r}, \quad 1 \leq r < \infty, \\
  \|f\|_{\infty,c} = \inf \bigl\{ K \geq 0 \colon \text{$|f(x)x^c| \leq K$ for a.e. $x \in \mathbb R^n_+$} \bigr\},
\end{gather*}
where $I = (1,\ldots,1)$. We denote by $L^r_c(\mathbb R^n_+)$ the Banach space of real-valued measurable functions $f$ on $\mathbb R^n_+$ with finite norm $\|f\|_{r,c}$. In particular, $L^1_I(\mathbb R^n_+) = L^1(\mathbb R^n_+)$, $L^2_{I/2}(\mathbb R^n_+) = L^2(\mathbb R^n_+)$ and $L^\infty_0(\mathbb R^n_+) = L^\infty(\mathbb R^n_+)$. Note that $(L^r_c(\mathbb R^n_+))^* = L^q_{I-c}(\mathbb R^n_+)$, where $1 \leq r < \infty$, $1/r + 1/q = 1$, $c \in \mathbb R^n_+$.

The spaces $L^r_c(\mathbb R^n_+)$ are analogous to the spaces $L^r(\mathbb R^n)$ in the case of harmonic analysis in $\mathbb R^n_+$, where the role of the Fourier transform is played by the Mellin transform. Recall that the Mellin transform of function $f$ is defined by the formula
\begin{equation}\label{mr.defM}
    (M f)(z) = \int_{\mathbb R^n_+} x^{z-I} f(x) \, dx.
\end{equation}
It follows from \eqref{in.defq}, \eqref{in.lvlqbnd} that $(M e^{-q})(z)$ is well-defined for $z \in \mathbb C^n$, $\Re z \in \mathbb R^n_+$.

As the following proposition shows, it is natural to consider the mapping properties of operators $R_q$, $R^h_q$ and $\Pi_q$ in the spaces $L^r_c(\mathbb R^n_+)$.
\begin{proposition}\label{prop.spaces} Let $q$ satisfy \eqref{in.defq} and \eqref{in.lvlqbnd}. Let $f \in L^r_{I-c}(\mathbb R^n_+) \cap C(\mathbb R^n_+)$ for some $c = (c_1,\ldots,c_n) \in \mathbb R^n_+$ and $1 \leq r \leq \infty$. Suppose that $h \in L^1_\alpha(\mathbb R^1_+)$, where $\alpha = c_1+\cdots+c_n$. Then the following estimates are valid:
\begin{align}
  \|R_q f\|_{r,c} & \leq \Gamma(\alpha)^{-1} \|e^{-q}\|_{1,c} \|f\|_{r,I-c}, \label{mr.Rqineq}\\
  \|R^h_q f\|_{r,c} & \leq \Gamma(\alpha)^{-1} \|e^{-q}\|_{1,c} \|h\|_{1,\alpha} \|f\|_{r,I-c},\label{mr.Rhqineq}\\
  \|(\Pi_q f)'\|_{r,c} & \leq p_0^{\alpha+1}\Gamma(\alpha+2)^{-1} \|e^{-q}\|_{1,c} \|f\|_{r,I-c}, \label{mr.Piqineq}
\end{align}
where $p_0 > 0$ is fixed, $(\Pi_q f)' = \Pi_q f(p_0,\cdot)$ and $\Gamma$ is the gamma function. Furthermore, if $r \in \{1,2\}$, then for a.e. $z \in \mathbb C^n$, $\Re z = c$, the following formulas are valid:
\begin{align}
    (MR_q f)(z) & = \Gamma(s)^{-1} (Mf)(I-z) \cdot (Me^{-q})(z), \label{mr.MfMRq} \\
    (MR^h_q f)(z) & =  \Gamma(s)^{-1}(M f)(I-z) \cdot (Me^{-q})(z) \cdot (Mh)(s), \label{mr.MfMRhq}\\
    (M(\Pi_q f)')(z) & = p_0^{s+1}\Gamma(s+2)^{-1} (Mf)(I-z) \cdot (Me^{-q})(z), \label{mr.MfMPiq}
\end{align}
where $s = z_1+\cdots+z_n$.
\end{proposition}
It follows from Proposition \ref{prop.spaces} that if $q$ satisfies \eqref{in.defq} and \eqref{in.lvlqbnd}, and $h \in L^1_c(\mathbb R^n_+)$, $c \in \mathbb R^n_+$, then operators $R_q$ and $R^h_q$ are linear continuous operators from $L^r_{I-c}(\mathbb R^n_+)$ to $L^r_c(\mathbb R^n_+)$, $1 \leq r \leq \infty$. Proposition \ref{prop.spaces} is proved in Section \ref{sec.ps}.

Note that formulas \eqref{mr.MfMRq}--\eqref{mr.MfMPiq} go back to \cite{Agal2014aen}.

For brevity, we will use the following terminology. Let $S \subset \mathbb C^n$ and let $H$ be a plane in $\mathbb C^n$. We say that:
\begin{enumerate}
 \item $S$ is $1$-meagre in $H$ iff $S \cap H$ in nowhere dense in $H$.
 \item $S$ is $2$-meagre in $H$ iff $S \cap H$ has measure zero in $H$.
 \item $S$ is $\infty$-meagre in $H$ iff $S$ does not intersect $H$.
\end{enumerate}

The main results of the present article can be summarized in the following Theorems \ref{thm.injPiq} and \ref{thm.injRhq}. 
\begin{theorem}\label{thm.injPiq} Let $q$ satisfy \eqref{in.defq}, \eqref{in.lvlqbnd} and let $c \in \mathbb R^n_+$, $r \in \{1,2,\infty\}$. Then $\Pi_q$ is injective in $L^r_{I-c}(\mathbb R^n_+)$ iff $R_q$ is injective in $L^r_{I-c}(\mathbb R^n_+)$. Furthermore, $R_q$ is injective in $L^r_{I-c}(\mathbb R^n_+)$ iff the set of zeros of $(Me^{-q})(z)$ is $r$-meagre in the plane $\Re z = c$.
\end{theorem}
Theorem \ref{thm.injPiq} characterizes industries in the framework of the generalized Houthakker--Johansen model, for which the profit function $\Pi_q$ uniquely determines the distribution of capacities over technologies $f$. This characterization is given in terms of production technologies described by cost function $q$. Theorem \ref{thm.injPiq} is proved in Section \ref{sec.tPiq}.

Theorem \ref{thm.injPiq} for operator $\Pi_q$ is a particular case of the following general theorem for operators $R^h_q$.
\begin{theorem}\label{thm.injRhq} Let $q$ satisfy \eqref{in.defq}, \eqref{in.lvlqbnd} and let $c \in \mathbb R^n_+$, $r \in \{1,2,\infty\}$. Let $h \in L^1_\alpha(\mathbb R^1_+)$ (and, additionally, $h \in L^2_\alpha(\mathbb R^1_+)$ if $r = 2$), where $\alpha = c_1+\cdots+c_n$. Then $R^h_q$ is injective in $L^r_{I-c}(\mathbb R^n_+)$ iff the set of zeros of $(Me^{-q})(z)$ is $r$-meagre in the plane $\Re z = c$ and the set of zeros of $(Mh)(s)$ is $r$-meagre on the line $\Re s = \alpha$.
 \end{theorem}
Theorem \ref{thm.injRhq} is proved in Section \ref{sec.tRhq}.

Next, we show that injectivity of operators $\Pi_q$ is preserved by composition of functions $q$ in the sense of \eqref{in.defnq}.

\begin{proposition}\label{prop.nest} Let $q \colon \mathbb R^{k+1}_+ \to \mathbb R^1_+$, $\phi \colon \mathbb R^m_+ \to \mathbb R^1_+$, where $k \geq 1$, $m \geq 2$. Suppose that $q$ and $\phi$ both satisfy \eqref{in.defq}, \eqref{in.lvlqbnd}, $\Pi_q$ is injective in $L^r_{I-c'}(\mathbb R^k_+)$, $\Pi_\phi$ is injective in $L^r_{I-d}(\mathbb R^m_+)$ for some $r \in \{1,2,\infty\}$, $c' = (c,d_1+\cdots+d_m) \in \mathbb R^k_+ \times \mathbb R^1_+$, $d = (d_1,\ldots,d_m) \in \mathbb R^m_+$. Let $\widetilde q(x,y) = q(x,\phi(y))$, $x \in \mathbb R^k_+$, $y \in \mathbb R^m_+$. Then $\widetilde q$ satisfies \eqref{in.defq}, \eqref{in.lvlqbnd} and $\Pi_{\widetilde q}$ is injective in $L^r_{I-c''}(\mathbb R^{k+m})$, where $c'' = (c,d)$.
\end{proposition}
In particular, we have the following corollary. Define a nested CES function recursively as follows:
\begin{enumerate}
 \item Every CES function $q_\alpha$, $\alpha \in (0,1]$, defined in \eqref{in.defCES} is a nested CES function.
 \item If $q$ and $\phi$ are nested CES functions, then $\widetilde q$ defined by formula \eqref{in.defnq} is a nested CES function.
\end{enumerate}
As it was mentioned in the introduction, nested CES functions were introduced in \cite{Sato1967} as an important generalization of popular CES functions, allowing different elasticities of substitution between different production factors.

\begin{corollary}\label{cor.nCES} Let $q$ be a nested CES function. Then $\Pi_q$ is injective in $L^r_{I-c}(\mathbb R^n_+)$ for any $c \in \mathbb R^n_+$ and $r \in \{1,2,\infty\}$. 
\end{corollary}
Proposition \ref{prop.nest} and Corollary \ref{cor.nCES} are proved in Section \ref{sec.pn}.

\section{Proof of Proposition \ref{prop.spaces}}\label{sec.ps}
\textit{Coarea formula.} The proof of Proposition \ref{prop.spaces} is based on the coarea formula, which is a generalization of the Fubini theorem to the case of curvilinear coordinates. The history of the coarea formula goes back to \cite{Kron1950en}. The coarea formula for the general case of Lipschitz continuous coordinates can be found, e.g., in \cite[Theorem 3.2.12]{Feder1969}.

If $q$ satisfies \eqref{in.defq}, \eqref{in.lvlqbnd} and $u \in L^1(\mathbb R^1_+)$, the coarea formula can be stated in the following form:
\begin{equation}\label{aux.coarea}
    \int_{\mathbb R^n_+} u(x) \, dx = \int_0^\infty t^{-1} (R_q u)\bigl(\tfrac{p}{t}\bigr) \, dt, \quad p \in \mathbb R^n_+.
\end{equation}

\textit{Formula \eqref{mr.MfMRq}.} Making a change of variables, we obtain the equality
\begin{equation*}
  I \overset{\text{def}}{=} \int\limits_{\mathbb R^n_+} f(x) \int\limits_{\mathbb R^n_+} p^{z-I} \exp(-q_p(x)) \, dp \, dx = \int\limits_{\mathbb R^n_+} x^{-z} f(x) \, dx \int\limits_{\mathbb R^n_+} y^{z-I} e^{-q(y)} \, dy.
\end{equation*}
On the other hand, using the Fubini theorem and the coarea formula \eqref{aux.coarea}, we obtain
\begin{equation*}
  I = \int\limits_{\mathbb R^n_+} p^{z-I} \int\limits_0^\infty t^{-1}e^{-t} (R_q f)\bigl(\tfrac p t\bigr) \, dt \, dp = \int\limits_0^{\infty} t^{s-1} e^{-t} \, dt \int\limits_{\mathbb R^n_+} p^{z-I} (R_q f)(p) \, dp.
\end{equation*}
Thus, the two above expressions for $I$ are equal, yielding formula \eqref{mr.MfMRq}.

\textit{Formulas \eqref{mr.MfMRhq}, \eqref{mr.MfMPiq}.} We need the following formula of \cite[Lemma 1]{Agal2015den}:
\begin{equation}\label{aux.Mhqp}
  \int\limits_{\mathbb R^n_+} p^{z-I} h(q_p(x)) \, dp \, \Gamma(s) = x^{-z} \int\limits_{\mathbb R^n_+} p^{z-I} e^{-q(p)}\,dp \int\limits_0^\infty t^{s-1} h(s) \, ds. 
\end{equation}
Integrating this formula with weight $f = f(x)$ over $\mathbb R^n_+$, we obtain formula \eqref{mr.MfMRhq}.

Formula \eqref{mr.MfMRhq} reduces to formula \eqref{mr.MfMPiq} if we fix $p_0 > 0$, set $h(t) = \max\{0,p_0 - t\}$, and note that
\begin{equation}\label{aux.MhPi}
  \int_0^\infty p^{s-1} h(s) \, ds = \frac{p_0^{s+1}}{s(s+1)}, \quad \Re s > 0.
\end{equation}

\textit{Estimate \eqref{mr.Rqineq}.} Let $1 \leq r < \infty$. Using the coarea formula \eqref{aux.coarea}, we obtain the equality
\begin{equation}\label{aux.MqtoRq}
  \int_{\mathbb R^n_+} p^{z-I} h(q(p)) \, dp = \int_0^\infty t^{s-1} h(t) \, dt \, (R_q x^{z-I})(I).
\end{equation}
Next, using the Jensen inequality we obtain the following estimate:
\begin{equation}\label{aux.RqJens}
  \begin{gathered}
       \bigl| (R_q f)(p) \bigr|^r \leq \bigl| (R_q x^{c-I})(p) \bigr|^{r-1} (R_q |f|^r x^{(r-1)(I-c)})(p) \\
    = p^{-c(r-1)} \bigl| (R_q x^{c-I})(I) \bigr|^{r-1} (R_q |f|^r x^{(r-1)(I-c)})(p) \\
    \overset{\eqref{aux.MqtoRq}}{=\joinrel=} \frac{\|e^{-q}\|^{r-1}_{1,c}}{\Gamma(\alpha)^{r-1}} p^{-c(r-1)} (R_q |f|^r x^{(r-1)(I-c)})(p).
  \end{gathered}
\end{equation}
Using this estimate we obtain the following inequality proving \eqref{mr.Rqineq}:
\begin{equation*}
 \begin{gathered}
    \|R_q f\|^r_{r,c} \overset{\eqref{aux.RqJens}}{\leq} \frac{\|e^{-q}\|^{r-1}_{1,c}}{\Gamma(\alpha)^{r-1}} \int\limits_{\mathbb R^n_+} p^{c-I} (R_q |f|^r x^{(r-1)(I-c)})(p) \, dp \\
    \overset{\eqref{mr.MfMRq}}{=\joinrel=} \frac{\|e^{-q}\|^r_{1,c}}{\Gamma(\alpha)^r} \int\limits_{\mathbb R^n_+} x^{r(I-c)-I} |f(x)|^r \, dx.
 \end{gathered}
\end{equation*}
The following estimate proves \eqref{mr.Rqineq} for $r = \infty$:
\begin{equation}\label{aux.Rqestinf}
  \begin{gathered}
    |p^c (R_q f)(p)| \leq \|f\|_{\infty,I-c} \bigl| p^{-c} (R_q x^{c-I})(p) \bigr| \\
     = \|f\|_{\infty,I-c} \bigl| (R_q x^{c-I})(I) \bigr| \overset{\eqref{aux.MqtoRq}}{=} \Gamma(\alpha)^{-1} \|e^{-q}\|_{1,c} \|f\|_{\infty,I-c}.
  \end{gathered}
\end{equation}

\textit{Estimates \eqref{mr.Rhqineq} and \eqref{mr.Piqineq}.} In order to prove \eqref{mr.Rhqineq} for $1 \leq r < \infty$, we need the following estimate based on the estimate \eqref{aux.RqJens} and on the coarea formula:
\begin{equation}\label{aux.RhqJens}
  \begin{gathered}
    \bigl|(R^h_q f)(p)\bigr|^r = \biggl| \int_0^\infty (R_q f)\bigl(\tfrac p t\bigr) t^{-1} h(t) \, dt \biggr|^r \\
    \leq \|h\|_{1,\alpha}^{r-1} \int_0^\infty \bigl|(R_q f)(\tfrac p t)\bigr|^r t^{\alpha(1-r)-1} |h(t)| \, dt \\
    \overset{\eqref{aux.RqJens}}{\leq} \|h\|^{r-1}_{1,\alpha} \frac{\|e^{-q}\|^{r-1}_{1,c}}{\Gamma(\alpha)^{r-1}} p^{-c(r-1)} \int\limits_{\mathbb R^n_+} t^{-1} \bigl(R_q |f|^r x^{(r-1)(I-c)} \bigr)(\tfrac p t)|h(t)| \, dt \\
  = \|h\|^{r-1}_{1,\alpha} \frac{\|e^{-q}\|^{r-1}_{1,c}}{\Gamma(\alpha)^{r-1}} p^{-c(r-1)} \bigl(R^{|h|}_q |f|^r x^{(r-1)(I-c)}\bigr)(p).
 \end{gathered}
\end{equation}
The following inequality proves \eqref{mr.Rhqineq} for $1 \leq r < \infty$:
\begin{equation*}
  \begin{gathered}
    \|R^h_q f\|^r_{r,c} \overset{\eqref{aux.RhqJens}}{\leq}  \|h\|^{r-1}_{1,\alpha} \frac{\|e^{-q}\|^{r-1}_{1,c}}{\Gamma(\alpha)^{r-1}} \int\limits_{\mathbb R^n_+} p^{c-I} \bigl(R^{|h|}_q |f|^r x^{(r-1)(I-c)}\bigr)(p) \, dp \\
    \overset{\eqref{mr.MfMRhq}}{=\joinrel=} \|h\|^r_{1,\alpha}\frac{\|e^{-q}\|^r_{1,c}}{\Gamma(\alpha)^r} \int\limits_{\mathbb R^n_+} x^{r(I-c) - I} |f(x)|^r \, dx.
  \end{gathered}
\end{equation*}
The inequality \eqref{mr.Rhqineq} for $r = \infty$ follows from the following estimate:
\begin{equation*}
  \begin{gathered}
    \bigl| p^c (R^h_q f)(p) \bigr| \leq \int_0^\infty \bigl| \bigl(\tfrac{p}{t}\bigr)^c (R_q f)\bigl(\tfrac p t\bigr) \bigr| t^{\alpha-1} |h(t)| \, dt \\
    \leq \|R_q f\|_{\infty,c} \|h\|_{1,\alpha} \overset{\eqref{aux.Rqestinf}}{\leq} \Gamma(\alpha)^{-1} \|e^{-q}\|_{1,c} \|h\|_{1,\alpha} \|f\|_{\infty,I-c}.
  \end{gathered}
\end{equation*}

\section{Wiener approximation theorems and their analogs}\label{sec.aux}
In the present section we prove two auxilary results that will be used in the proof of Theorem \ref{thm.injPiq}. These results are based on the following two Wiener's approximation theorems. Let $\mathcal F$ and $\mathcal F^{-1}$ denote the Fourier transform and the inverse Fourier transform, respectively:
\begin{align}
    \mathcal Ff(\xi) & = (2\pi)^{-\frac n 2} \int_{\mathbb R^n} e^{-i\xi x} f(x) \, dx, \quad \xi \in \mathbb R^n, \label{aux.defF}\\
    \mathcal F^{-1}f(\xi) & = (2\pi)^{-\frac n 2} \int_{\mathbb R^n} e^{i\xi x} f(x) \, dx, \quad \xi \in \mathbb R^n. \label{aux.defFi}
\end{align}
For a given function $f$ on $\mathbb R^n$ denote by $\mathcal S_f$ the linear span if its additive shifts:
\begin{equation}\label{aux.defS}
  \mathcal S_f = \lspan \bigl\{ f_a \mid f_a(x) = f(x-a), \; a \in \mathbb R^n \bigr\}.
\end{equation}

\begin{theorem}\label{aux.WienL2} Let $f \in L^2(\mathbb R^n)$. Then $\mathcal S_f$ is dense in $L^2(\mathbb R^n)$ iff $\mathcal F f \neq 0$ a.e.
\end{theorem}

\begin{theorem}\label{aux.WienL1} Let $f \in L^1(\mathbb R^n)$. Then $\mathcal S_f$ is dense in $L^1(\mathbb R^n)$ iff $\mathcal F f$ does not vanish. 
\end{theorem}
The proofs of these theorems for $n=1$ can be found in \cite{Wien1933}. These proofs of Theorem \ref{aux.WienL2} and of the ``only if'' part of Theorem \ref{aux.WienL1} also work for $n \geq 2$. On the other hand, we don't have a reference for a proof of the ``if'' part of Theorem \ref{aux.WienL1} for $n \geq 2$. For the reader's convenience, we sketch a proof below.

\begin{proof}[Sketch of the proof of the ``if'' part of Theorem \ref{aux.WienL1}]
We need to show that any $h \in L^1(\mathbb R^n)$ can be approximated in $L^1(\mathbb R^n)$ by the shifts of $f$. Without loss of generality, we suppose that $\mathcal F h$ has compact support. It follows from the Hahn--Banach theorem that it is sufficient to show that for any $K \in L^\infty(\mathbb R^n)$ the equality $f*K = 0$ implies $h * K = 0$, where $*$ stands for convolution. Hence, it is sufficient to show that there exists $g \in L^1(\mathbb R^n)$ such that $f*g = h$.

We use the theory of commutative Banach algebras, see, e.g., \cite{Kain2009} for definitions. Let $L^1(\cdot,\mathbb C)$ be the space of complex-valued Lebesgue integrable functions.

Let $\Omega \subset \mathbb R^n$ be an open bounded set with closure $\overline \Omega$ and containing $\supp \mathcal F h$. Note that $L^1(\mathbb R^n,\mathbb C)$ is a commutative Banach algebra with repsect to convolution. Then $I = \bigl\{ g \in L^1(\mathbb R^n,\mathbb C) \mid \mathcal F g \equiv 0$ on $\overline \Omega \bigr\}$ is a closed ideal in $L^1(\mathbb R^n,\mathbb C)$. Put $A = L^1(\mathbb R^n,\mathbb C) / I$. Note that $A$ is a commutative Banach algebra with unit $e + I$, where $e \in L^1(\mathbb R^n)$ is an arbitrary function such that $\mathcal F e \equiv 1$ on $\overline \Omega$ and $e+I$ denotes the coset of $e$ in $A$.

One can show that the only non-zero complex multiplicative linear functionals on $A$ are of the form $\varphi_\xi \colon a+I \mapsto a(\xi)$, where $\xi \in \overline \Omega$, $a \in L^1(\mathbb R^n,\mathbb C)$. Using \cite[Theorem 1.2.9]{Kain2009} and taking into account that $\varphi_\xi(f+I) \neq 0$ for any $\xi \in \overline \Omega$, we obtain that $f+I$ is invertible in $A$. It means that there exists $g_0 \in L^1(\mathbb R^n,\mathbb C)$ such that $\mathcal F f \cdot \mathcal Fg_0 \equiv 1$ on $\overline \Omega$. Put $g = h* \Re g_0$. Then $f * g = h$.
\end{proof}

The main results of the present section are given in the following lemmas. For a given function $k$ on $\mathbb R^n_+$ we denote by $\mathcal T_k$ the linear span of its multiplicative shifts:
\begin{equation}\label{aux.defT}
  \mathcal T_k = \lspan \bigl\{ k_p \mid k_p(x) = k(p_1 x_1,\ldots, p_n x_n), \; p \in \mathbb R^n_+\bigr\}.
\end{equation}

\begin{lemma}\label{lemma.L2} Let $k \in L^2_c(\mathbb R^n_+)$. Then the following statements are equivalent:
\begin{enumerate}
 \item[(1)] $\mathcal T_k$ is dense in $L^2_c(\mathbb R^n_+)$.
 \item[(2)] $(Mk)(z) \neq 0$ a.e. for $\Re z = c$.
 \item[(3)] The equation 
\begin{equation*}
    \int_{\mathbb R^n_+} k_p(x) f(x) \, dx = 0, \quad p \in \mathbb R^n_+,
\end{equation*}
    has only the trivial solution $f = 0$ in $L^2_{I-c}(\mathbb R^n_+)$.
\end{enumerate}
\end{lemma}

\begin{proof} (1 $\Longleftrightarrow$ 2). For a given function $f$ on $\mathbb R^n_+$ and $c \in \mathbb R^n$ we define the function $E_c f$ on $\mathbb R^n$ as follows:
\begin{equation}
  (E_c f)(y) = e^{c \cdot y} f(e^y), \quad y = (y_1,\ldots,y_n) \in \mathbb R^n, \label{aux.defE}
\end{equation}
where $e^y = \exp(y) =  (e^{y_1},\ldots,e^{y_n})$. One can see that
\begin{equation}\label{aux.Eckp}
  (E_c k_{\exp(a)})(y) = e^{-c \cdot a}(E_c k)(y+a), 
\end{equation}
where $a = (a_1,\ldots,a_n) \in \mathbb R^n$ and $k_{\exp(a)}(x) = k(e^{a_1}x_1,\ldots,e^{a_n}x_n)$. It follows from \eqref{aux.Eckp} that
\begin{equation}\label{aux.EcTc}
  E_c \mathcal T_k = \mathcal S_{E_c k},
\end{equation}
where the sets $\mathcal S_{E_c k}$ and $E_c \mathcal T_k$ are defined in \eqref{aux.defS} and \eqref{aux.defT}, respectively.

One can see that
\begin{equation}
  \text{$E_c$ is an isometry from $L^r_c(\mathbb R^n_+)$ to $L^r(\mathbb R^n)$ for any $1 \leq r \leq \infty$}. \label{aux.Eciso}
\end{equation}
It follows from \eqref{aux.EcTc} and \eqref{aux.Eciso} that $\mathcal T_k$ is dense in $L^2_c(\mathbb R^n_+)$ if and only if $\mathcal S_{E_c k}$ is dense in $L^2(\mathbb R^n)$. According to Theorem \ref{aux.WienL2}, this is equivalent to
\begin{equation}\label{aux.FiEcknz}
  (\mathcal F^{-1} E_c k)(\xi) \neq 0 \quad \text{for a.e. $\xi \in \mathbb R^n$}.
\end{equation}

Now, using definition \eqref{mr.defM}, we obtain the equality
\begin{equation}\label{aux.MviaF}
  T_c M = (2\pi)^{\frac n 2}\mathcal F^{-1} E_c,
\end{equation}
where $T_c$ maps a function $\varphi$ on $c+i\mathbb R^n$ to the function $T_c \varphi$ on $\mathbb R^n$ defined as
\begin{equation}\label{aux.defTc}
    (T_c \varphi)(\xi) = \varphi(c+i\xi), \quad \xi \in \mathbb R^n.
\end{equation}
It follows from  \eqref{aux.MviaF} that \eqref{aux.FiEcknz} is equivalent to $(Mk)(z) \neq 0$ a.e. for $\Re z = c$.

(1 $\Longleftrightarrow$ 3). This is a consequence of the Hahn--Banach theorem.
\end{proof}

\begin{lemma}\label{lemma.Linf}
    Let $k \in L^1_c(\mathbb R^n_+)$. Then the following statements are equivalent:
\begin{enumerate}
    \item[(1)] $\mathcal T_k$ is dense in $L^1_c(\mathbb R^n_+)$.
    \item[(2)] $(M k)(z) \neq 0$ for $\Re z = c$.
    \item[(3)] The equation
    \begin{equation}\label{aux.Linfeq}
        \int_{\mathbb R^n_+} k_p(x) f(x) \, dx = 0, \quad p \in \mathbb R^n_+
    \end{equation}
    has only the trivial solution $f = 0$ in the class $L^\infty_{I-c}(\mathbb R^n_+)$.
\end{enumerate}
\end{lemma}

\begin{proof} (2 $\Longleftrightarrow$ 3). Let $E_c$ be defined by formula \eqref{aux.defE}. One can see that
\begin{equation}\label{aux.EcPars}
  \begin{gathered}
  \int_{\mathbb R^n_+} u(x)v(x) \, dx = \int_{\mathbb R^n} (E_c u)(y) \, (E_{I-c} v)(y) \, dy,\\
   \text{for any $u \in L^r_c(\mathbb R^n_+)$, $v \in L^p_{I-c}(\mathbb R^n_+)$, $r \geq 1$, $p \geq 1$, $\tfrac 1 r + \tfrac 1 p = 1$.}
  \end{gathered}
\end{equation}
It follows from formulas \eqref{aux.Eckp}, \eqref{aux.Eciso} and \eqref{aux.EcPars} that equation \eqref{aux.Linfeq} has only the trivial solution $f = 0$ in $L^\infty_{I-c}(\mathbb R^n_+)$ if and only if the equation
\begin{equation}
  \int_{\mathbb R^n} (E_c k)(y-a) \Phi(y) \, dy = 0, \quad a \in \mathbb R^n,
\end{equation}
has only the trivial solution $\Phi = 0$ in $L^\infty(\mathbb R^n)$. By the Hahn--Banach theorem it is equivalent to density of $\mathcal S_{E_c k}$ in $L^1(\mathbb R^n)$. By Theorem \ref{aux.WienL1} and formula \eqref{aux.MviaF} it is equivalent to $(Mk)(z) \neq 0$ for $\Re z = c$.
 
(1 $\Longleftrightarrow$ 3). This is a consequence of the Hahn--Banach theorem.
\end{proof}

\section{Proof of Theorem \ref{thm.injPiq}}\label{sec.tPiq}
We are going to prove Theorem \ref{thm.injPiq} for operator $R_q$. The statement concerning operator $\Pi_q$ is a corollary of Theorem \ref{thm.injRhq} and formula \eqref{aux.MhPi}. For brevity, we denote
\begin{equation*}
  H^n_c = \bigl\{ z \in \mathbb C^n \colon \Re z = c \bigr\}, \quad c \in \mathbb R^n.
\end{equation*}
For a given function $\varphi$ defined on $H^n_c$ we set
\begin{equation*}
  Z_c(\varphi) = \bigl\{ z \in H^n_c \colon \varphi(z) = 0 \bigr\}.
\end{equation*}

($r = 1$). ($\implies$). Suppose that, on the contrary, there exists a bounded relatively open subset $U \subset H^n_c$, $U \neq \varnothing$, such that $(Me^{-q})(z) = 0$ for all $z \in U$.

Let $\chi \in C^\infty(H^n_{I-c})$ be a non-zero function such that $\chi(I-z) = 0$ for $z \not\in U$. Define the function $\widehat \chi$ on $\mathbb R^n_+$ by the following formula:
\begin{equation}\label{tpr.chidef}
  \widehat \chi = E^{-1}_{(I-c)} \mathcal F T_{(I-c)} \chi,
\end{equation}
where the operators $E_{(I-c)}$ and $T_{(I-c)}$ are defined in formulas \eqref{aux.defE}, \eqref{aux.defTc} and $\mathcal F$ is the Fourier transform defined in formula \eqref{aux.defF}.

Using that $\mathcal F C^\infty_c(\mathbb R^n) \subset L^1(\mathbb R^n)$ and taking into account \eqref{aux.Eciso}, we obtain
\begin{equation}
  \widehat \chi \in L^1_{I-c}(\mathbb R^n_+) \;\; \text{and} \;\; \widehat \chi \neq 0.
\end{equation}
It follows from formulas \eqref{aux.MviaF} and \eqref{tpr.chidef} that 
\begin{equation}\label{tpr.Mchi}
  (M \widehat \chi)(I-z) = (2\pi)^{\frac n 2}\chi(I-z), \quad  z \in H^n_c.
\end{equation}
Using formula \eqref{mr.MfMRq} with $f = \widehat \chi$ and taking into account that $\chi(I-z) = 0$ for $z \not\in U$ and $(Me^{-q})(z) = 0$ for $z \in U$, we obtain that $(MR_q \widehat \chi)(z) = 0$ for $z \in H^n_c$. It follows that $R_q \widehat \chi = 0$ which contradicts the injectivity of $R_q$ on $L^1_{I-c}(\mathbb R^n_+)$.

($r = 1$). ($\Longleftarrow$). Suppose that zeros of function $Me^{-q}$ are nowhere dense in the plane $H^n_c$. We are going to show that $R_q$ is injective in $L^1_{I-c}(\mathbb R^n_+)$ by contradiction. 

Suppose that there exists $f \in L^1_{I-c}(\mathbb R^n_+)$ such that $f \neq 0$, $R_q f = 0$. The following statements hold true:
\begin{gather}
  \text{$H^n_{I-c} \setminus Z_{I-c}(Mf)$ is open in $H^n_{I-c}$ and not empty,} \label{tpr.Zcf}\\
  \text{$H^n_c \setminus Z_c\bigl(Me^{-q}\bigr)$ is open and dense in $H^n_c$.} \label{tpr.ZcMhMeq}
\end{gather}

It follows from \eqref{tpr.Zcf} and \eqref{tpr.ZcMhMeq} that there exists a relatively open set $U \subset H^n_c$, $U \neq \varnothing$, such that
\begin{equation}\label{tpr.MfMhMeqnz}
  (Mf)(I-z) \, (Me^{-q})(z) \neq 0, \quad z \in U.
\end{equation}
Formulas \eqref{mr.MfMRq} and \eqref{tpr.MfMhMeqnz} imply that $(MR_q f)(z) \neq 0$ for $z \in U$. It contradicts the assumption that $R_q f = 0$.

($r=2$). ($\implies$). Suppose, on the contrary, that there exists a bounded set $U \subset H^n_c$ of positive Lebesgue measure on $H^n_c$, such that $(Me^{-q})(z) = 0$ for all $z \in U$. Define function $\chi$ on $H^n_{I-c}$ by the formula
\begin{equation*}
    \chi(I-z) = \begin{cases}
              1, & z \in U, \\
              0, & z \not\in U. 
            \end{cases}
\end{equation*}
Next, define function $\widehat \chi$ on $\mathbb R^n_+$ using formula \eqref{tpr.chidef}. It follows from formulas \eqref{aux.Eciso}, \eqref{tpr.chidef} and from the inclusion $\mathcal FL^2(\mathbb R^n) \subset L^2(\mathbb R^n)$ that
\begin{equation*}
  \widehat \chi \in L^2_{I-c}(\mathbb R^n_+) \;\; \text{and} \;\; \widehat \chi \neq 0.
\end{equation*}
It also follows from \eqref{aux.MviaF} and \eqref{tpr.chidef} that formula \eqref{tpr.Mchi} holds.

Using formula \eqref{mr.MfMRq} with $f = \widehat \chi$ and taking into account that $(M\widehat \chi)(I-z) \cdot (Me^{-q})(z) = 0$ for all $z \in H^n_c$, we obtain that $(MR_q \widehat \chi)(z) = 0$ for $z \in H^n_c$. Hence, $R_q \widehat \chi = 0$, and this is a contradiction to the assumption that $R_q$ is injective in $L^2_{I-c}(\mathbb R^n_+)$.

($r=2$.) ($\Longleftarrow$). Suppose that $Me^{-q}$ does not vanish on the plane $H^n_c$ a.e. Let $f \in L^2_{I-c}(\mathbb R^n_+)$ be a function such that $R_q f = 0$.

Note that $x^{I-2c}f(x) \in L^2_c(\mathbb R^n_+)$. Using Lemma~\ref{lemma.L2} we find $a_k \in \mathbb R$, $p_k \in \mathbb R^n_+$ such that
\begin{equation}\label{tpr.fL2repr}
  x^{I-2c} f(x) = \sum\nolimits_{k=1}^\infty a_k \exp\bigl(-q_{p_k}(x)\bigr) \quad \text{in $L^2_c(\mathbb R^n_+)$.}
\end{equation}
Using formula \eqref{tpr.fL2repr} and the coarea formula \eqref{aux.coarea}, we obtain the following equality showing that $f = 0$:
\begin{equation}
  \|f\|^2_{2,I-c} = \int_{\mathbb R^n_+} x^{I-2c}f^2(x) \, dx = \sum_{k=1}^\infty a_k \int_0^\infty t^{-1} e^{-t} (R_q f)\bigl(\tfrac{p_k}{t}\bigr) \, dt = 0.
\end{equation}

($r = \infty$). ($\implies$). Assume, on the contrary, that there exists $z^0 = (z^0_1,\ldots,z^0_n) \in H^n_c$ such that $(Me^{-q})(z^0) = 0$. Set $\widehat \chi(x) = x^{z^0-I}$.

Note that $\widehat \chi \in L^\infty_{I-c}(\mathbb R^n_+)$ and that
\begin{equation}
  (R_q \widehat \chi)(p) = p^{-z^0} (R_q \widehat \chi)(I) \overset{\eqref{aux.MqtoRq}}{=\joinrel=} p^{-z^0} \frac{(Me^{-q})(z^0)}{\Gamma(z^0_1+\cdots+z^0_n)} = 0, \quad p \in \mathbb R^n_+.
\end{equation}
This equality contradicts the injectivity of $R_q$ on $L^\infty_{I-c}(\mathbb R^n_+)$.

($r=\infty$). ($\Longleftarrow$). Let $f \in L^\infty_{I-c}(\mathbb R^n_+)$ be a function such that $R_q f = 0$. Note that $x^{2(I-c)}e^{-|x|}f(x) \in L^1_c(\mathbb R^n_+)$. Using Lemma \ref{lemma.Linf} we find $a_k \in \mathbb R$, $p_k \in \mathbb R^n_+$ such that
\begin{equation*}
    x^{2(I-c)}e^{-|x|}f(x) = \sum\nolimits_{k=1}^\infty a_k \exp\bigl(-q_{p_k}(x)\bigr) \quad \text{in $L^1_c(\mathbb R^n_+)$}.
\end{equation*}
Hence,
\begin{equation*}
    \int_{\mathbb R^n_+} x^{2(I-c)} e^{-|x|} f^2(x)\, dx = \sum_{k=1}^\infty a_k \int_0^\infty t^{-1} e^{-t} (R_q f)\left(\tfrac{p_k}{t}\right) \, dt = 0.
\end{equation*}
This equality implies that $f = 0$.

\section{Proof of Theorem \ref{thm.injRhq}}\label{sec.tRhq}
($r = 1$). ($\implies$). Suppose, on the contrary, that $R^h_q$ is injective in $L^1_{I-c}(\mathbb R^n_+)$ but either $Z_c(Me^{-q})$ has non-empty relative interior in $H^n_c$, or $Z_\alpha(Mh)$ has non-empty relative interior in $H^1_\alpha$. It implies that there exists a bounded subset $U \subset Z_c( (Me^{-q}) (Mh)')$ with non-empty relative interior in $H^n_c$, where $(Mh)'(z) = (Mh)(z_1+\cdots+z_n)$, $z = (z_1,\ldots,z_n)$.

Choose any $\chi \in C^\infty(H^n_{I-c})$ such that $\chi \neq 0$, $\chi(I-z) = 0$ for $z \not\in U$, and define $\widehat \chi \in L^1_{I-c}(\mathbb R^n_+)$ by the formula \eqref{tpr.chidef}. Repeating the proof of statement ($\implies$) of Theorem \ref{thm.injPiq} for the case of $r = 1$ (and using formula \eqref{mr.MfMRhq} instead of formula \eqref{mr.MfMRq}), one can show that $R^h_q \widehat \chi = 0$. It contradicts the injectivity of $R^h_q$ in $L^1_{I-c}(\mathbb R^n_+)$.

($r=1$). ($\Longleftarrow$). To prove the other direction, it is sufficient to show that if $R_q$ is injective in $L^1_{I-c}(\mathbb R^n_+)$ and $Z_c(Mh)$ is nowhere dense in $H^1_\alpha$, then $R^h_q$ is injective in $L^1_{I-c}(\mathbb R^n_+)$ (compare the statements of Theorems \ref{thm.injPiq} and \ref{thm.injRhq} for $r=1$). 

Suppose that $R_q$ is injective in $L^1_{I-c}(\mathbb R^n_+)$ and $Z_c(Mh)$ is nowhere dense in $H^1_\alpha$. Let $f \in L^1_{I-c}(\mathbb R^n_+)$ and assume that $R^h_q f = 0$.

Using formulas \eqref{mr.MfMRq} and \eqref{mr.MfMRhq} we obtain the identity
\begin{equation}\label{tpr.MRhqtoMRq}
  (MR^h_q f)(z) = (MR_q f)(z) \, (Mh)(s),
\end{equation}
where $z = (z_1,\ldots,z_n) \in H^n_c$, $s = z_1+\cdots+z_n$.

It follows from formulas \eqref{aux.MviaF} and \eqref{tpr.MRhqtoMRq} that $MR_q f = 0$ in $H^1_c$ as a continuous function vanishing on an open dense subset. Hence, $R_q f = 0$ and $f = 0$.

($r=2$). It follows from Lemma \ref{lemma.L2} that $R^h_q$ is injective in $L^2_{I-c}(\mathbb R^n_+)$ iff
\begin{equation}\label{tpr.Mhqnz}
    \int_{\mathbb R^n_+} x^{z-I} h(q(x)) \, dx \neq 0 \quad \text{for a.e. $z \in H^n_c$}.\\
\end{equation}
Using formula \eqref{aux.Mhqp} with $x = I$ we obtain that \eqref{tpr.Mhqnz} holds true for a.e. $z \in H_c^n$ iff $Me^{-q}$ does not vanish a.e. in $H^n_c$ and $Mh$ does not vanish a.e. in $H^1_\alpha$.

($r=\infty$). It follows from Lemma \ref{lemma.Linf} that $R^h_q$ is injective in the space $L^\infty_{I-c}(\mathbb R^n_+)$ iff the inequality of \eqref{tpr.Mhqnz} holds for all $z \in H^n_c$. Using formula \eqref{aux.Mhqp} with $x = I$, one can see that \eqref{tpr.Mhqnz} is true for all $z \in H^n_c$ iff $Me^{-q}$ does not vanish in $H^n_c$ and $Mh$ does not vanish in $H^1_\alpha$.

\section{Proof of Proposition \ref{prop.nest} and Corollary \ref{cor.nCES}}\label{sec.pn}
\begin{proof}[Proof of Proposition \ref{prop.nest}] Using the coarea formula \eqref{aux.coarea}, we obtain the following equality:
\begin{equation}\label{tpr.MetqRphi}
  \begin{gathered}
    (Me^{-\widetilde q})(z,w) \overset{\text{def}}{=} \int_{\mathbb R^k_+}\int_{\mathbb R^m_+} x^{z-I} y^{w-I} \exp(-q(x,\phi(y)) \, dx\,dy \\
  = \int_{\mathbb R^k_+}\int_0^\infty x^{z-I} t^{s-1} \exp(-q(x,t))\,dx\,dt \, (R_\phi y^{w-I})(I),
  \end{gathered}
\end{equation}
where $z \in \mathbb C^k$, $\Re z = c$, $w = (w_1,\ldots,w_m) \in \mathbb C^m$, $\Re w = d$. Using formula \eqref{aux.MqtoRq} with $h(t) = e^{-t}$, $p = I$, and equality \eqref{tpr.MetqRphi}, and taking into account that $(Me^{-t})(s) = \Gamma(s)$, we obtain
\begin{equation}\label{tpr.Metq}
  (Me^{-\widetilde q})(z,w) = \frac{(Me^{-\phi})(w) \, (Me^{-q})(z,s)}{\Gamma(s)}, \quad s = w_1+\cdots+w_m.
\end{equation}
Proposition \ref{prop.nest} is a corollary of Theorem \ref{thm.injPiq} and formula \eqref{tpr.Metq}. 
\end{proof}

\begin{proof}[Proof of Corollary \ref{cor.nCES}] We set $l(y) = y_1+\cdots+y_n$ for $y = (y_1,\ldots,y_n)$. Let $z = (z_1,\ldots,z_n) \in \mathbb C^n$, $\Re z \in \mathbb R^n_+$, $s = z_1+\cdots+z_n$, $\alpha \in (0,1]$. Using the coarea formula \eqref{aux.coarea}, we obtain the following chain of equalities:
\begin{equation*}
 \begin{gathered}
    (Me^{-q_\alpha})(z) = \int_{\mathbb R^n_+} x^{z-I} \exp\bigl( -C ( a_1 x_1^\alpha + \cdots + a_n x_n^\alpha)^\frac{1}{\alpha} \bigr) \, dx \\
    = \tfrac{a_1^{-\frac{z_1}{\alpha}} \cdots a_n^{-\frac{z_n}{\alpha}}}{\alpha^n C^s} \int_{\mathbb R^n_+} y^{\frac z \alpha - I} \exp\bigl(- (y_1+\cdots+y_n)^\frac{1}{\alpha} \bigr) \, dy \\
   = \tfrac{a_1^{-\frac{z_1}{\alpha}} \cdots a_n^{-\frac{z_n}{\alpha}}}{\alpha^{n-1} C^s} \int_0^\infty t^{s - 1}\exp(-t) \, dt  \, (R_l y^{\frac z \alpha - I})(I).
 \end{gathered}
\end{equation*}
Using this formula and formula \eqref{aux.MqtoRq} with $q(x)=l(x)$, $h(t) = e^{-t}$, and taking into account that $(Me^{-l})(z)=\Gamma(z_1)\cdots\Gamma(z_n)$, we obtain that
\begin{equation}\label{tpr.MqCES}
    (Me^{-q_\alpha})(z) = \frac{\Gamma(s)}{\alpha^{n-1}C^s\Gamma(\tfrac{s}{\alpha})}\prod\limits_{j=1}^n a_j^{-\frac{z_j}{\alpha}} \Gamma(\tfrac{z_j}{\alpha}).
\end{equation}
Corollary \ref{cor.nCES} follows directly from Theorem \ref{thm.injPiq}, Proposition \ref{prop.nest} and formula \eqref{tpr.MqCES}, if we take into account that $(Me^{-q_\alpha})(z) \neq 0$ for $\Re z \in \mathbb R^n_+$.
\end{proof}

\section{Aknowledgements}
The author would like to thank Prof. A.\,A. Shananin and Prof. R.\,G. Novikov for their helpful comments and suggestions.

The present work is supported by RFBR grant №14-07-00075 А.

\bibliographystyle{IEEEtranS}
\bibliography{biblio_utf}

\end{document}